\documentclass{article}[12pt]
\usepackage[vlined, ruled, linesnumbered]{algorithm2e}
\usepackage{amsmath}

\usepackage{graphicx}
\usepackage[T1]{fontenc}
\usepackage[cp1250]{inputenc}
\usepackage{amsthm}
\usepackage{amssymb,amsmath,amsfonts}
\usepackage[hmargin=2.5cm, vmargin=3cm, hcentering]{geometry}
\usepackage{appendix}
\usepackage{caption}
\usepackage{enumitem}
\usepackage{indentfirst}

\usepackage{setspace}
\newtheorem{thm}{Theorem}%[section]
\newcounter{cptRef}
\newcommand\myAffiliation[1]{\ref{refAffiliation:#1}}
\renewcommand\author[3]{{\em #1 #2}\ $^{\mbox{\scriptsize\myAffiliation{#3}}}$}
\newcommand\affiliationWithoutReferenceToAuthor[1]{#1}
\newcommand\authorWithoutReferenceToAffiliation[2]{{\em #1 #2}}
\renewcommand\title[1]{{\Large #1}}

\newenvironment{titleArea}{\begin{center}}{\end{center}}

\newcommand{\rr}{\mathbb{R}}

\newcommand{\todo}[1]{{\bf \texttt{TODO: #1}}}

\theoremstyle{plain}

\theoremstyle{definition}

\begin{document}

\begin{titleArea}

\title{

First non-trivial upper bound on the circular chromatic number of the plane.
}

\medskip

%% Authors
%\authorWithoutReferenceToAffiliation{Konstanty}{Junosza-Szaniawski},
%\authorWithoutReferenceToAffiliation{Joanna}{SokĂłĹ‚} and
%\authorWithoutReferenceToAffiliation{Krzysztof}{WÄ™sek}

%% One or more affiliations
%\affiliationWithoutReferenceToAuthor{Faculty of Mathematics and Information Science,\\ Warsaw University of Technology, Poland}

%% Authors
\authorWithoutReferenceToAffiliation{Konstanty}{Junosza-Szaniawski},

\medskip

%% One or more affiliations

\affiliationWithoutReferenceToAuthor{Faculty of Mathematics and Information Science,\\ Warsaw University of Technology, Poland}

\end{titleArea}

\begin{abstract}
We consider circular version of the famous Nelson-Hadwiger problem. It is know that 4 colors are necessary and 7 colors suffice to 
color the euclidean plane in such a way that points at distance one get different colors. In $r$-circular coloring we assign arcs of length one of a circle with a perimeter $r$ in such a way that points at distance one get disjoint arcs. In this paper we show the existence of $r$-circular coloring for  $r=4+\frac{4\sqrt{3}}{3}\approx 6.30$. It is the first result with $r$-circular coloring of the plane with $r$ smaller than 7. We also show $r$-circular coloring of the plane with $r<7$ in the case when we require disjoint arcs for points at distance belonging to the internal $[\frac{10}{11}, \frac{12}{11}]$. 
\end{abstract}

\section{Introduction}

We refer to the famous Nelson-Hadwiger problem and a well studied coloring model - circular coloring. 
The Nelson Hadwiger problem is the question for the chromatic number of the euclidean  plane, which is the minimum number of colors required to color the plane in such a way  that no two points at distance 1 from each other have the same color. The exact number is not known. We only know that at least 4 colors are needed \cite{moser} and 7 colors suffice \cite{fisher}. For a comprehensive history of the Hadwiger-Nelson problem see a monograph \cite{soifer}.

An $r$-circular coloring of a graph $G=(V,E)$ is a function $c:V\to [0,r)$  such that for any edge $uv$ of $G$ holds $1\le |c(u)-c(v)|\le r-1$. Notice that an $r$-circular coloring can be seen as an assignment of arcs of length 1 of a circle with perimeter $r$ to vertices of $G$ in such a way that adjacent vertices get disjoint arcs. The circular chromatic number of a graph $G$ is the number  $\chi_c(G)=\inf\{r\in \rr: \text{ there exists } r\text{-circular coloring of } G\}$.   Circular coloring was first introduced by Vince \cite{vince}. For a survey see paper by Zhu \cite{zhu}. Circular coloring applies in scheduling theory, to minimize the average time of process, that are repeated many times. It is known (\cite{vince}) that the circular chromatic number does not exceed chromatic number, but is bigger than the chromatic number minus one, formally $\chi_c(G)\in (\chi(G)-1, \chi(G)]$ for any graph $G$. 

In this paper we consider the circular coloring of the plane. Precisely  the circular coloring of infinite graph $G_0$, with the set of all points in the plane as the vertex set and the set of all pairs of points at distance one as the set of edges.  If we combine known bounds for the chromatic number of the graph $G_0$ with properties of the circular chromatic number we obtain $3<\chi_c(G_0)\le 7$. The lower bound can be improved by comparing with another interesting parameter - fractional chromatic number. To define it we need to define first $j$-fold coloring of the plane. A function is a $j$-fold coloring if it assigns a $j$-element set of natural numbers to the vertices of a graph $G$ in such a way that adjacent vertices get disjoint sets. The fractional chromatic number of  $G$ is defined by 

\[\chi_f(G)=\inf\{\frac{k}{j}: \text{ there exists } j\text{-fold coloring of }G \text{ using }k\text{ colors.}\}\]

It is known \cite{zhu} that the fractional  chromatic number does not exceed circular chromatic number. The best known lower bound for the  fractional chromatic number of $G_0$ is $\frac{32}{9}\approx 3.55$ and it can be found in the book of Scheinerman and Ullman \cite{fgt}. This gives us a lower bound on the circular chromatic number of the plane.  DeVos, Ebrahimi, Ghebleh,  Goddyn,  Mohar,  Naserasr \cite{circ} improved this bound by showing that the chromatic number of the plane is at least 4. In this paper we give first non-trivial upper bound on the circular chromatic number of the plane.

Exoo considered more restricted coloring of the plane in \cite{exoo}. He asked for the minimum number of colors to color the plane in such a way that any points in at distance belonging to a given interval $[1-\epsilon, 1+\epsilon]$ get different colors. For $\epsilon=0$ the problem reduces to Nelson-Hadwiger problem. The fractional and the $j$-fold Exoo type coloring was studied in \cite{my}.  The method of circular coloring of the plane presented in this paper can be adapted to Exoo type coloring of the plane.

\section{Main results}

Let us introduce some  definitions and notation. For $(x_1,y_1),(x_2,y_2)\in \rr^2$ by $d((x_1,y_1),(x_2,y_2))=\sqrt{(x_1-x_2)^2+(y_1-y_2)^2}$ we denote the euclidean distance. 
For $x\in \rr$ and $\ell\in \rr_+$ we define $\lfloor x\rfloor _\ell=\lfloor \frac{x}{\ell}\rfloor \cdot \ell$ and $(x)_\ell=x-\lfloor x\rfloor _\ell$. Notice that for $\ell =1$ function $\lfloor x\rfloor_\ell$ is the standard floor function  $\lfloor x \rfloor$.  
By $G_\varepsilon$ we denote an infinite graph with the vertex set $\rr^2$ and the set of edges $\{(x_1,y_1)(x_2,y_2): d((x_1,y_1),(x_2,y_2))\in [1-\varepsilon,1+\varepsilon]\}$. Our main result is the first non trivial bound on the circular chromatic number of the plane in a sense defined in Nelson-Hadwiger problem i.e. on $\chi_c(G_0)$.

\begin{thm}\label{main}\label{zero}
\[\chi_c(G_0)\le 4+\frac{4\sqrt{3}}{3}\approx 6.3095\]
\end{thm}

\begin{proof}

We not only give the bound on the circular chromatic number but actually  present the $(4+\frac{4\sqrt{3}}{3})$-circular coloring of the plane.
Let $\ell=2+2\sqrt{3}$ and let $r=\frac{2\sqrt{3}}{3}\ell=4+\frac{4\sqrt{3}}{3}$.

Let us start with some intuition on the construction  a $r$-circular coloring of the plane. 
Let $R$ denote a rectangle $[0, \ell)\times [0,\frac{1}{2})$.  We define $r$-circular coloring of the rectangle $R$ by $c(x,y)=\frac{2\sqrt{3}}{3}x$, (where $(x,y)\in R$). Then we extend this coloring in a circular way on a strip $S=\rr\times [0,\frac{1}{2})$. We simply join copies of the rectangle $R$ by their vertical sides so they form a strip $S$. 
Each copy of $R$ is colored in the same way as the original rectangle $R$. 
Then we take copies of the strip $S$ and join them with horizontal sides.  Each strip $S$ is colored in the same way, but we  shift each copy of $S$  by $(1+\frac{\sqrt{3}}{2})$ to the right comparing with the strip below (see Figures \ref{kontur}-\ref{czb}).

\begin{center}

\begin{figure}[h]
\includegraphics[width=18cm]{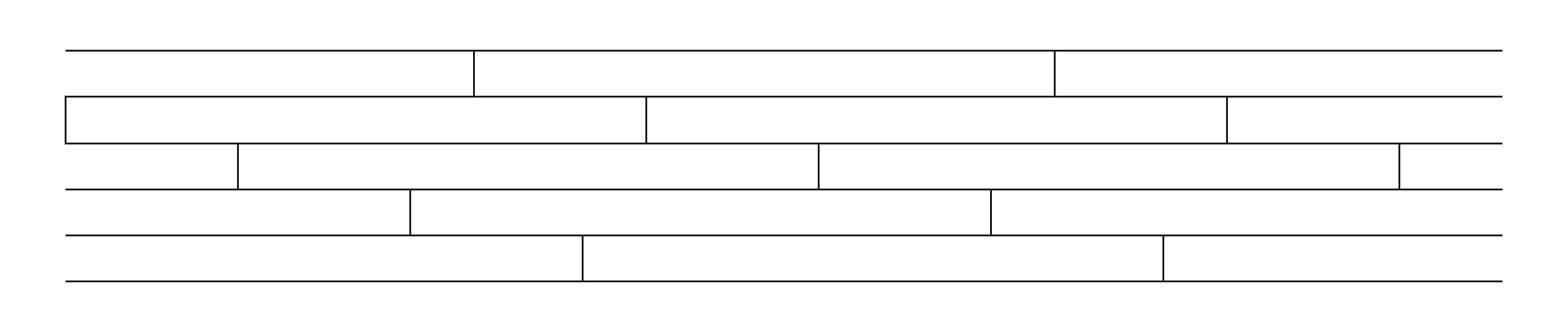}
\caption{Partition of the plane into copies of rectangle $R$} \label{kontur}
\end{figure}

\begin{figure}[h]
\includegraphics[width=18cm]{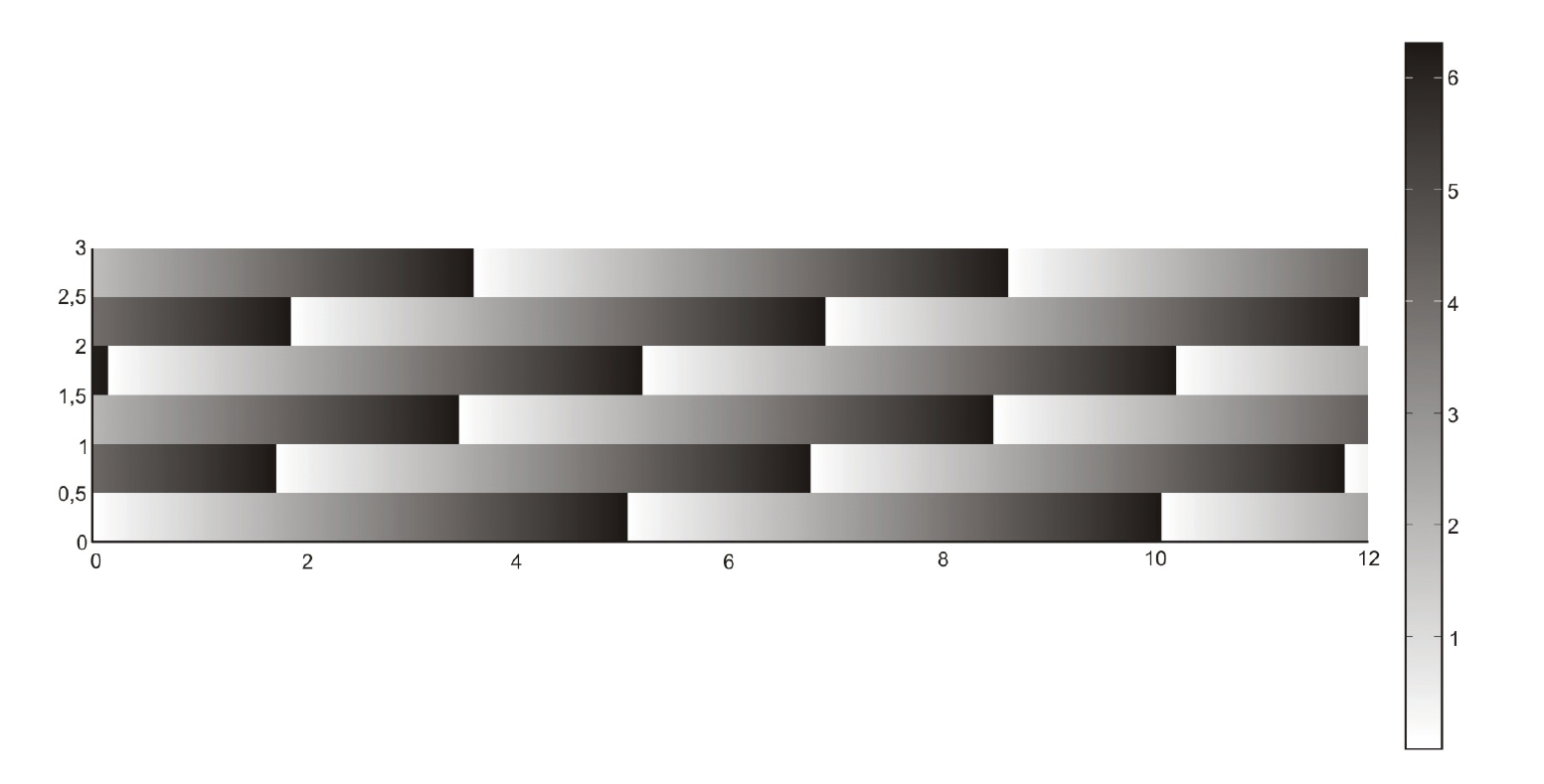} 
\caption{$(4+\frac{4\sqrt{3}}{3})$-circular coloring of the plane, black and white}\label{czb}
\end{figure}

\begin{figure}[h]
\includegraphics[width=18cm]{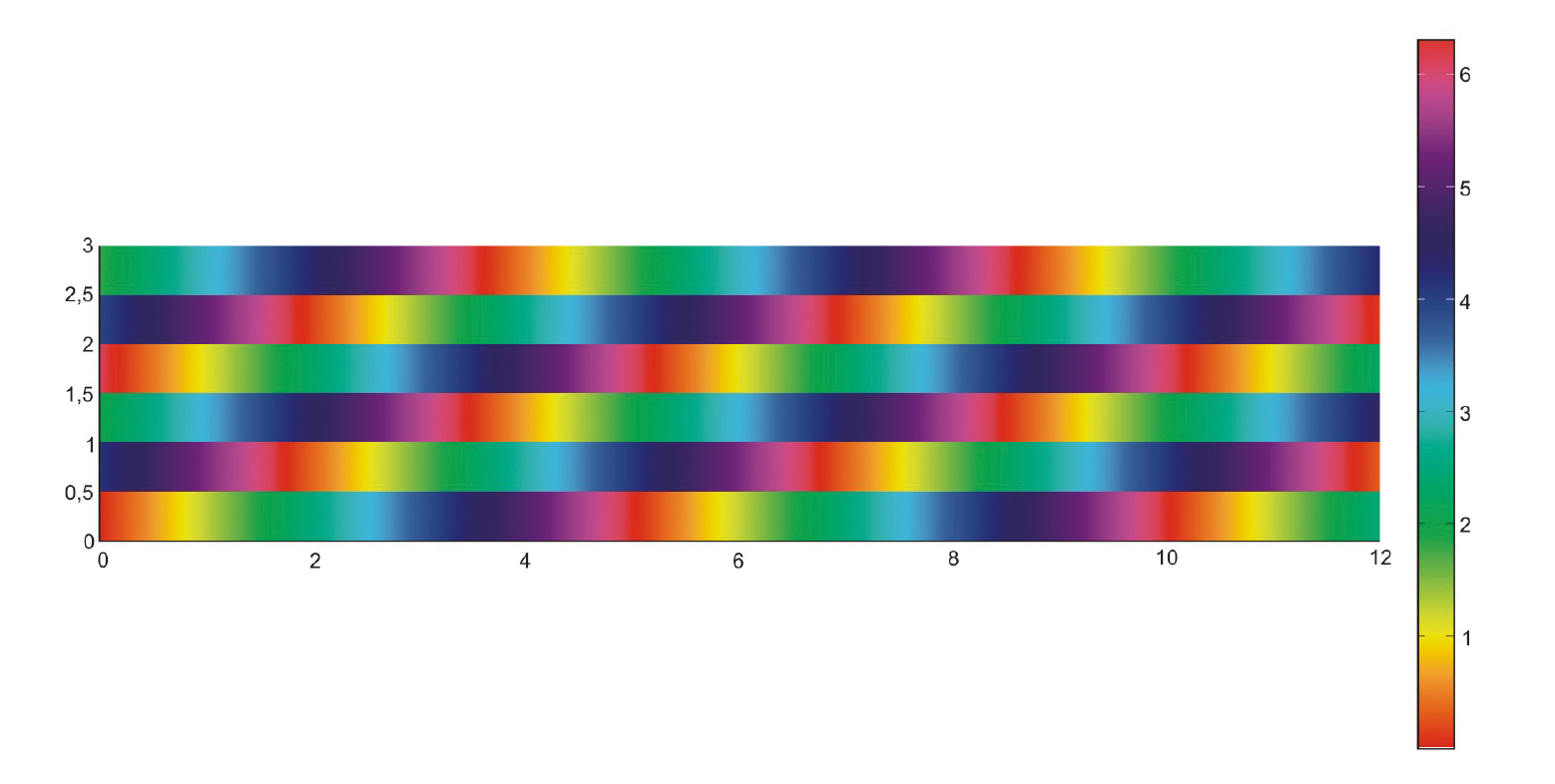}
\caption{$(4+\frac{4\sqrt{3}}{3})$-circular coloring of the plane, color } \label{kolor}
\end{figure}

\end{center}

 Formally the $r$-circular coloring of the graph $G_0$ is defined by:   
 
$$c(x,y)=\frac{2\sqrt{3}}{3}\left(x -(2+\sqrt{3})\lfloor y\rfloor_{\frac{1}{2}} \right)_\ell$$

To prove that $c$ is a circular coloring of $G_0$  we show that for any pair of points $(x_1,y_1)$ and $(x_2,y_2)$ at distance 1 the the condition $1\le |c(x_1,y_1)-c(x_2,y_2)|\le r-1$ is fulfilled. Without loss of generality we assume that  $c(x_1,y_1)\ge c(x_2,y_2)$ i.e $(x_1-(2+\sqrt{3})\lfloor y_1\rfloor _\frac{1}{2})_\ell \ge (x_2-(2+\sqrt{3})\lfloor y_2\rfloor _\frac{1}{2})_\ell$. Let $(x'_1,y'_1)=(x_1-x_2,y_1-\lfloor y_2\rfloor _\frac{1}{2})$ and $(x'_2,y'_2)=(0,y_2-\lfloor y_2\rfloor _\frac{1}{2})$. Notice that $d((x'_1,y'_1),(x'_2,y'_2))=d((x_1,y_1),(x_2,y_2))$. 

Moreover using fact that fact that $\lfloor y_2-\lfloor y_1\rfloor_\frac{1}{2}\rfloor _\frac{1}{2}=\lfloor y_2\rfloor _\frac{1}{2}-\lfloor y_1\rfloor_\frac{1}{2}$ and $(a-b)_\ell=(a)_\ell-(b)_\ell$ for $(a)_\ell\ge (b)_\ell$ we obtain: 

$$|c(x'_1,y'_1)-c(x'_2,y'_2)|= \frac{2\sqrt{3}}{3}\left|(x_1-x_2-(2+\sqrt{3})\lfloor y_1-\lfloor y_2\rfloor _\frac{1}{2} \rfloor_\frac{1}{2})_\ell-0\right|=$$
$$=\frac{2\sqrt{3}}{3}\left|(x_1-x_2-(2+\sqrt{3})(\lfloor y_1 \rfloor_\frac{1}{2}-\lfloor y_2\rfloor _\frac{1}{2}))_\ell\right|=
\frac{2\sqrt{3}}{3}\left|(x_1-(2+\sqrt{3})\lfloor y_1 \rfloor_\frac{1}{2}-(x_2-(2+\sqrt{3}))\lfloor y_2\rfloor _\frac{1}{2}))_\ell\right|=$$
$$=\frac{2\sqrt{3}}{3}\left|(x_1-(2+\sqrt{3})\lfloor y_1\rfloor _\frac{1}{2}))_\ell-(x_2-(2+\sqrt{3})\lfloor y_2 \rfloor_\frac{1}{2})_\ell\right|=\left|c(x_1,y_1)-c(x_2,y_2)\right|$$

Therefore  without loss of generality we may assume that $x_2=0$ and $y_2\in [0,\frac{1}{2})$. Recall that $d((x_1,y_1),(x_2,y_2))=1$.

It remains to prove that $|c(x_1,y_1)-c(x_2,y_2)|=|c(x_1,y_1)|\subseteq [1,3+\frac{4\sqrt{3}}{3}]$.  Consider following cases:

Case 1: $\lfloor y_1\rfloor _\frac{1}{2}=0, x_1>0$.  
In this case $x_1\in (\frac{\sqrt{3}}{2},1]$ and $c(x_1,y_1)\in (1,\frac{2\sqrt{3}}{3}]\subseteq [1,3+\frac{4\sqrt{3}}{3}]$

Case 2: $\lfloor y_1\rfloor _\frac{1}{2}=0, x_1<0$.  
In this case $x_1\in [-1, -\frac{\sqrt{3}}{2})$ and $c(x_1,y_1)\in [4+\frac{2\sqrt{3}}{3}, 3+\frac{4\sqrt{3}}{3})\subseteq [1,3+\frac{4\sqrt{3}}{3}]$

Case 3: $\lfloor y_1\rfloor _\frac{1}{2}=\frac{1}{2}$.  
In this case $x_1\in [-1, 1]$ and $c(x_1,y_1)\in [3,3+\frac{4\sqrt{3}}{3} ]\subseteq [1,3+\frac{4\sqrt{3}}{3}]$

Case 4: $\lfloor y_1\rfloor _\frac{1}{2}=1$.  
In this case $x_1\in [-\frac{\sqrt{3}}{2},\frac{\sqrt{3}}{2}]$ and $c(x_1,y_1)\in [1,\frac{5}{2}]\subseteq [1,3+\frac{4\sqrt{3}}{3}]$

Case 5: $\lfloor y_1\rfloor _\frac{1}{2}=-\frac{1}{2}$.  
In this case $x_1\in [-1, 1]$ and $c(x_1,y_1)\in [1,1+\frac{4\sqrt{3}}{3}]\subseteq [1,3+\frac{4\sqrt{3}}{3}]$

Case 6: $\lfloor y_1\rfloor _\frac{1}{2}=-1$.  
In this case $x_1\in [-\frac{\sqrt{3}}{2},\frac{\sqrt{3}}{2}]$ and $c(x_1,y_1)\in [1+\frac{4\sqrt{3}}{3},\frac{3}{2}+\frac{4\sqrt{3}}{3}]\subseteq [1,3+\frac{4\sqrt{3}}{3}]$

% 
%\begin{center}
%\begin{figure}[h]
%
%\begin{picture}(100,100)(-50,00)
%
%\put(0,0){\line(1,0){350}}
%\put(0,15){\line(1,0){350}}
%\put(0,30){\line(1,0){350}}
%\put(0,45){\line(1,0){350}}
%\put(0,60){\line(1,0){350}}
%\put(0,75){\line(1,0){350}}
%\put(0,90){\line(1,0){350}}
%
%\put(0,0){\line(0,1){15}}
%\put(165,0){\line(0,1){15}}
%\put(330,0){\line(0,1){15}}
%\put(55,15){\line(0,1){15}}
%\put(220,15){\line(0,1){15}}
%\put(110,30){\line(0,1){15}}
%\put(275,30){\line(0,1){15}}
%\put(0,45){\line(0,1){15}}
%\put(165,45){\line(0,1){15}}
%\put(330,45){\line(0,1){15}}
%\put(55,60){\line(0,1){15}}
%\put(220,60){\line(0,1){15}}
%\put(110,75){\line(0,1){15}}
%\put(275,75){\line(0,1){15}}
%
%\end{picture}
%\caption{Partitions of the plane into copies of the rectangle $R$}\label{paski}
%\end{figure}\end{center}
%
%\begin{center}
%
%\begin{figure}[h]
%
%\includegraphics[width=18cm]{plot3.png} \label{figure}
%\caption{$(4+\frac{4\sqrt{3}}{3})$-circular coloring of the plane}
%\end{figure}
%\end{center}
%

\end{proof}

The method of  circular coloring defined in Theorem \ref{main} can be also used to color in a circular way  graphs $G_\varepsilon$. For small values of $\varepsilon$ the number $\chi_c(G_\varepsilon)$ is strictly smaller than the known bound on the  classic chromatic number of $G_\varepsilon$.

\begin{thm} For $\varepsilon\le\frac{1}{23}(53 - 6 \sqrt{71})\approx 0.1062$ holds 
\[\chi_c(G_\varepsilon)\le 4+\frac{4+4\varepsilon}{\sqrt{3\varepsilon^2-10\varepsilon+3}}<7\]
\end{thm}

\begin{proof} The proof is analogous to the prove of Theorem \ref{zero}. We partition the plane into copies of a rectangle, and color each copy in the same way. The difference is the size of rectangle and the factor in the coloring function. Let  $\ell=2+2\varepsilon+2\sqrt{3\varepsilon^2-10\varepsilon+3}$, $r= 4+\frac{4+4\varepsilon}{\sqrt{3\varepsilon^2-10\varepsilon+3}}$.  To prove the Theorem we define a $r$-circular coloring of the $G_\epsilon$ by :   
 
$$c(x,y)=\frac{2}{\sqrt{3\varepsilon^2-10\varepsilon+3}}\left(x -\left(1+\frac{2\sqrt{3\varepsilon^2-10\varepsilon+3}}{1+\varepsilon}\right)\lfloor y\rfloor_{\frac{1+\varepsilon}{2}} \right)_\ell$$

 To make formulas shorter let $a=\frac{1}{2}\sqrt{3\varepsilon^2-10\varepsilon+3}$. Then $r=4+\frac{2+2\varepsilon}{a}$ and  $$c(x,y)=\frac{1}{a}\left(x-\left(1+\varepsilon+a\right)\frac{2}{1+\varepsilon}\lfloor y\rfloor_\frac{1+\varepsilon}{2}\right)_\ell.$$ 
 
To prove that $c$ is a circular coloring of $G_\varepsilon$  we show that for any pair of points $(x_1,y_1)$ and $(x_2,y_2)$ at distance in $[1-\varepsilon,1+\varepsilon]$ the condition  $1\le |c(x_1,y_1)-c(x_2,y_2)|\le r-1$ holds. Without loss of generality we assume that  $c(x_1,y_1)\ge c(x_2,y_2)$ i.e $$\left(x_1 -\left(1+\varepsilon+a\right)\frac{2}{1+\varepsilon}\lfloor y_1\rfloor_{\frac{1+\varepsilon}{2}} \right)_\ell\ge \left(x_2 -\left(1+\varepsilon+a\right)\frac{2}{1+\varepsilon}\lfloor y_2\rfloor_{\frac{1+\varepsilon}{2}} \right)_\ell$$

 Let $(x'_1,y'_1)=(x_1-x_2,y_1-\lfloor y_2\rfloor_{\frac{1+\varepsilon}{2}})$ and $(x'_2,y'_2)=(0,y_2-\lfloor y_2\rfloor_{\frac{1+\varepsilon}{2}})$. Notice that $d((x'_1,y'_1),(x'_2,y'_2))=d((x_1,y_1),(x_2,y_2))$. 

Moreover, similarly as in the proof above, one can prove $|c(x'_1,y'_1)-c(x'_2,y'_2)|= \left|c(x_1,y_1)-c(x_2,y_2)\right|$. 
Therefore  without loss of generality we may assume that $x_2=0$ and $y_2\in [0,\frac{1+\varepsilon}{2})$.

Recall that $d((x_1,y_1),(x_2,y_2))\in [1-\varepsilon,1+\varepsilon]$.
It remains to prove that $|c(x_1,y_1)-c(x_2,y_2)|=|c(x_1,y_1)|\subseteq [1,r-1]$.  Consider following cases:

Case 1: $\lfloor y_1\rfloor _\frac{1+\varepsilon}{2}=0, x_1>0$.  
In this case $x_1\in (a,1+\varepsilon]$ and $c(x_1,y_1)\in (1,\frac{1+\varepsilon}{a}]\subseteq [1,r-1]$

Case 2: $\lfloor y_1\rfloor _\frac{1+\varepsilon}{2}=0, x_1<0$.  
In this case $x_1\in [-1-\varepsilon, -a)$ and $c(x_1,y_1)\in [r-\frac{1+\varepsilon}{a}, r-1)\subseteq [1,r-1]$

Case 3: $\lfloor y_1\rfloor _\frac{1+\varepsilon}{2}=\frac{1+\varepsilon}{2}$.  
In this case $x_1\in [-(1+\varepsilon), 1+\varepsilon]$ and $c(x_1,y_1)\in [3, r-1]\subseteq [1,r-1]$

Case 4: $\lfloor y_1\rfloor _\frac{1+\varepsilon}{2}=1+\varepsilon$.  
In this case $x_1\in [-a,a]$ and $c(x_1,y_1)\in [1,3]\subseteq [1,r-1]$

Case 5: $\lfloor y_1\rfloor _\frac{1+\varepsilon}{2}=-\frac{1+\varepsilon}{2}$.  
In this case $x_1\in [-1-\varepsilon, 1+\varepsilon]$ and $c(x_1,y_1)\in [1,1+\frac{2(1+\varepsilon)}{a}]\subseteq [1,r-1]$

Case 6: $\lfloor y_1\rfloor _\frac{1+\varepsilon}{2}=-1$.  
In this case $x_1\in [-a,a]$ and $c(x_1,y_1)\in [1+\frac{2(1+\varepsilon)}{a},3+\frac{2(1+\varepsilon)}{a}]\subseteq [1,r-1]$

\end{proof}


\begin{thebibliography}{}
%
% and use \bibitem to create references. Consult the Instructions
% for authors for reference list style.
%
\bibitem{circ}
M. DeVos, J. Ebrahimi, M. Ghebleh, L. Goddyn, B. Mohar, R. Naserasr, Circular Coloring the Plane, SIAM Journal on Discrete Mathematics, 21, 2 (April 2007), 461-465. 

\bibitem{exoo} 
G. Exoo, $\varepsilon$-Unit Distance Graphs, Discrete Comput. Geom, 33: 117-123, 2005.

\bibitem{fisher} 
H. Hadwiger, Ungeloste Probleme, Elemente der Mathematik, 16: 103-104, 1961.

\bibitem{my} J.Grytczuk, K.Junosza-Szaniawski, J.Sokół, K.Węsek, Fractional and $j$-fold coloring of the plane, to appear. 


\bibitem{moser}
L. Moser, W. Moser, Problems for Solution, Canadian Bulletin of Mathematics, 4: 187-189, 1961.


\bibitem{fgt} 
E.R. Scheinerman, D.H. Ullman, Fractional Graph Theory, John Wiley and Sons, 2008

%\bibitem{hoch-odon} 
%R. Hochberg, P. O'Donnell, A Large Independent Set in the Unit Distance Graph, Geombinatorics, 3(4): 83-84, 1993.

%\bibitem{fish-ul} 
%D. Fisher, D. Ullman, The fractional chromatic number of the plane, Geombinatorics, 2(1): 8-12, 1992.

%\bibitem{ivanov}
%L.L. Ivanov, On the Chromatic Numbers of $\mathbb{R}^2$ and $\mathbb{R}^3$ with Intervals of Forbiden Distances, Electronic Notes in Discrete Mathematics, 29: 159-162, 2007.
%
%\bibitem{walczak}
%Z. Walczak, J.M. Wojciechowski, Transmission Scheduling in Packet Radio Networks using Graph Coloring Algorithm, Wireless and Mobile Communications, 2006. ICWMC~'06.

\bibitem{soifer}
A. Soifer, The Mathematical Coloring Book, Springer, 2008.

%\bibitem{woodall}
%D.R Woodall, Distances realized by sets covering the plane, Journal of Combinatorial Theory, Series A, 74(4): 279-286, 1997.

%\bibitem{falconer}
%K. J. Falconer, The realization of distances in measurable subsets covering $\mathbb{R}^n$, Journal of Combinatorial Theory, Series A, 31 (1981), 187-189.


\bibitem{vince}A. Vince. Star chromatic number. J. Graph Theory 12 (1988) 551-559

\bibitem{zhu} X. Zhu,  Circular chromatic number: a survey. Discrete Mathematics 229 (2001), 371-410. 

\end{thebibliography}
\end{document}